\documentclass[12pt]{article}
\usepackage{amssymb}
\usepackage{amsfonts}

\newtheorem{theorem}{Theorem}

\newtheorem{corollary}[theorem]{Corollary}

\newtheorem{example}[theorem]{Example}

\newtheorem{proposition}[theorem]{Proposition}

\newenvironment{proof}[1][Proof]{\noindent\textbf{#1.} }{\ \rule{0.5em}{0.5em}}
\input{tcilatex}
\begin{document}

\title{On subcompactness and countable subcompactness of metrizable spaces
in $\mathbf{ZF}$}
\author{Kyriakos Keremedis}
\maketitle

\begin{abstract}
We show in $\mathbf{ZF}$\ that:\smallskip\ 

(i) Every subcompact metrizable space is completely metrizable, and every
completely metrizable space is countably subcompact.\smallskip

(ii) A metrizable space $\mathbf{X}=(X,T)$ is countably compact iff it is
countably subcompact relative to $T$.\smallskip 

(iii) For every metric space $\mathbf{X}=(X,d),$ the following are
equivalent:\newline
(a) $\mathbf{X}$ is compact;\newline
(b) for every open filter $\mathcal{F}$ of $\mathbf{X},\dbigcap \{\overline{F%
}:F\in \mathcal{F}\}\neq \emptyset $; \newline
(c) $\mathbf{X}$ is subcompact relative to $T$.\ \smallskip 

We also show:\smallskip

(iv) The negation of each of the statements,

(a) every countably subcompact metrizable space is completely metrizable,

(b) every countably subcompact metrizable space is subcompact,

(c) every complete metrizable space is subcompact

\ is relatively consistent with $\mathbf{ZF}$.\smallskip

(v) $\mathbf{AC}$ iff for every family $\{\mathbf{X}_{i}:i\in I\}$ of
metrizable subcompact spaces, for every family $\{\mathcal{B}_{i}:i\in I\}$
such that for every $i\in I$, $\mathcal{B}_{i}$ is a subcompact base for $%
\mathbf{X}_{i}$, the Tychonoff product $\mathbf{X}=\dprod\limits_{i\in I}%
\mathbf{X}_{i}$ is subcompact with respect to the standard base $\mathcal{B}$
of $\mathbf{X}$ generated by the family $\{\mathcal{B}_{i}:i\in I\}$%
.\bigskip\ \newline

\end{abstract}

\section{Notation and Terminology}

Let $\mathbf{X}=(X,T)$ be a topological space and $\mathcal{H}$ be a
non-empty subset of $\mathcal{P}(X)\backslash \{\emptyset \}$. $\mathcal{H}$
is called a \textit{filterbase} of $X$ iff the intersection of any two
members of $\mathcal{H}$ contains an element of $\mathcal{H}$. A filterbase $%
\mathcal{H}$ of $X$ closed under supersets, i.e., for all $A\in \mathcal{H}$
and $B\in \mathcal{P}(X)$, if $A\subset B$ then $B\in \mathcal{H}$, is
called \textit{filter} of $X$. A filterbase (resp. filter) $\mathcal{H}$ $%
\subseteq T$\ of $X$ is called \textit{open} \textit{filterbase} (resp. 
\textit{open filter}) of $\mathbf{X}$. An open filter $\mathcal{H}$ of $%
\mathbf{X}$ is called \textit{total} iff for every $x\in X$, there exists a
neighborhood $V$ of $x$ such that $X\backslash \overline{V}\in \mathcal{H}$.

Assume $\mathbf{X}$ is $T_{3}$ and $\mathcal{B}$ is an open base of $\mathbf{%
X}$. An open filterbase $\mathcal{F}\subseteq \mathcal{B}$ of $\mathbf{X}$
is called \textit{regular }$\mathcal{B}$\textit{-filterbase} iff for every $%
F\in \mathcal{F}$ there exists $B\in \mathcal{F}$ with $\overline{B}%
\subseteq F$. If a regular $\mathcal{B}$-filterbase $\mathcal{F}$ is
countable, then $\mathcal{F}$ is called a \textit{countable regular} $%
\mathcal{B}$\textit{-filterbase}. In particular, the countable regular $%
\mathcal{B}$-filterbase $\mathcal{F}=\{F_{n}\in \mathcal{B}:n\in \mathbb{N}$%
, $\overline{F_{n+1}}\subseteq F_{n}\}$ is called \textit{a regular} $%
\mathcal{B}$\textit{-sequence}.

Let $\mathcal{U}$ be a family of subsets of $X$. An element $x\in X$ is
called a \textit{cluster} \textit{point} of $\mathcal{U}$ iff every
neighborhood of $x$ meets infinitely many members of $\mathcal{U}$. $%
\mathcal{U}$ is said to be \textit{locally finite} if $\mathcal{U}$ has no
cluster points.\smallskip

$\mathbf{X}$ is said to be \textit{compact} (resp. \textit{countably compact}%
) iff every open cover $\mathcal{U}$ of $\mathbf{X}$ (resp. countable open
cover $\mathcal{U}$ of $\mathbf{X}$) has a finite subcover $\mathcal{V}$.
Equivalently, $\mathbf{X}$ is compact (resp. countably compact) iff the
intersection of every family (resp. countable family) of closed sets of $%
\mathbf{X}$ with the finite intersection property (fip for abbreviation) is
non-empty.\smallskip 

$\mathbf{X}$ is said to be\ \textit{lightly compact} (resp. \textit{%
countably lightly compact}) iff $\mathbf{X}$ has no infinite (resp. no
countably infinite) locally finite families of open subsets.\smallskip\
Light compactness has been introduced in \cite{mp}. Countable light
compactness is condition $(B_{3})$ (: Every pairwise disjoint family $%
\mathcal{U}=\{U_{n}:n\in \mathbb{N}\}$ of non-empty open subsets of $\mathbf{%
X}$ has a cluster point in $\mathbf{X}$) in \cite{ad} and it is equivalent
to light compactness in $\mathbf{ZFC}$ (= Zermelo-Fraenkel set theory $%
\mathbf{ZF}$ together with axiom of choice $\mathbf{AC}$). Lightly compact
spaces are also called \textit{feebly compact}, see e.g. \cite{s}.\smallskip

A $T_{3}$ space $\mathbf{X}$ is called \textit{subcompact} (resp. \textit{%
countably subcompact}) if there exists an open base $\mathcal{B}$ such that
for every regular $\mathcal{B}$-filterbase (resp. for every countable
regular $\mathcal{B}$\textit{-}filterbase), $\dbigcap \{F:F\in \mathcal{F}%
\}\neq \emptyset $. The base $\mathcal{B}$ is called \textit{subcompact}
(resp. \textit{countably subcompact}). Subcompact and countably subcompact
spaces have been introduced and investigated in \cite{dg}.\smallskip

Let $\mathbf{X}=(X,d)$ be a metric space, $x\in X$ and $\varepsilon >0$. $%
B_{d}(x,\varepsilon )=\{y\in X:d(x,y)<\varepsilon \}$ (resp. $%
D_{d}(x,\varepsilon )=\{y\in X:d(x,y)\leq \varepsilon \}$) denotes the open
(resp. closed) ball in $\mathbf{X}$ with center $x$ and radius $\varepsilon $%
. If no confusion is likely to arise we shall omit the subscript $d$ from $%
B_{d}(x,\varepsilon )$ and $D_{d}(x,\varepsilon )$. Given $B\subseteq
X,B\neq \emptyset ,$ 
\[
\delta (B)=\sup \{d(x,y):x,y\in B\}\in \mathbb{R}_{+}\cup \{+\infty \}
\]%
will denote the \textit{diameter} of $B$\smallskip . $T_{d}$ will denote the
topology on $X$ produced by the family of all open discs of $\mathbf{X}$%
.\smallskip\ 

$\mathbf{X}$ is called \textit{subcompact} (resp. \textit{countably
subcompact}) iff the topological space $(X,T_{d})$ is subcompact (resp.
countably subcompact).\smallskip

A sequence of points $(x_{n})_{n\in \mathbb{N}}$ of $\mathbf{X}$ is called 
\textit{Cauchy} iff for every $\varepsilon >0$ there exists $n_{0}\in 
\mathbb{N}$ such that for all $n,m\geq n_{0},d(x_{n},x_{m})<\varepsilon $%
.\smallskip\ 

A metric space $(X,d)$ is said to be \textit{complete} iff every Cauchy
sequence in $X$ converges to some point of $X$.\smallskip

A \textit{completion} of $\mathbf{X}$ is a complete metric space $(Y,\rho )$
together with an isometric map $H:\mathbf{X}\rightarrow \mathbf{Y}$ such
that $\overline{H(X)}=Y$. It is a well-known $\mathbf{ZF}$ result that for
every $x_{0}\in X$ the mapping: 
\[
H:(X,d)\rightarrow (C_{b}(X,\mathbb{R}),\rho ),\text{ }H(x)=f_{x} 
\]%
where, $C_{b}(X,\mathbb{R})$ is the family of all bounded continuous
functions from $X$ to $\mathbb{R},$ $\rho $ is the sup metric ($\rho
(f,g)=\sup \{|f(x)-g(x)|:x\in X\}$) and for every $x\in X,$ $%
f_{x}:X\rightarrow \mathbb{R}$ is the function given by:%
\[
f_{x}(t)=d(x,t)-d(x_{0},t)\text{,} 
\]%
is such an isometric map. Thus, $\mathbf{Y}=(\overline{H(X)},\rho )$ is a
completion of $\mathbf{X}$.\smallskip

A topological space $\mathbf{X}=(X,T)$ is said to be \textit{completely
metrizable,} or \textit{topologically complete} iff there is a metric $%
d:X\times X\rightarrow \lbrack 0,\infty )$ such that $T_{d}=T$ and $(X,d)$
is a complete metric space.\smallskip

Let $X$ be an infinite set. We say that $X$ is \textit{Dedekind infinite}
(resp. \textit{weakly Dedekind infinite}) iff $X$ (resp. $\mathcal{P}(X)$)
has a countably infinite subset. Otherwise, $X$ is called \textit{Dedekind
finite} (resp. \textit{weakly Dedekind finite}).\smallskip

Below we list the weak forms of the axiom of choice we shall use in this
paper.

\begin{itemize}
\item $\mathbf{AC}$ : For every family $\mathcal{A}$ of non-empty sets there
exists a function $f$ such that for all $x\in \mathcal{A}$, $f(x)\in x$.

\item $\mathbf{CAC}$ : $\mathbf{AC}$ restricted to countable families.

\item $\mathbf{IDI}(\mathbb{R})$ : Every infinite subset $\mathbb{R}$ is
Dedekind infinite.

\item $\mathbf{IWDI}$ : Every infinite set is weakly Dedekind infinite.
\end{itemize}

For $\mathbf{ZF}$ models satisfying $\mathbf{CAC},$ $\mathbf{IDI}(\mathbb{R}%
) $, or their negations, we refer the reader to \cite{hr}.

\section{Introduction and some preliminary results}

In this paper, the intended context for reasoning will be $\mathbf{ZF}$. If
a statement is provable in $\mathbf{ZF}$ we will add $(\mathbf{ZF})$ in the
beginning of that statement. Otherwise, there will appear $(\mathbf{ZFC})$.

Most mathematicians are aware if $\mathbf{AC}$ is used in a proof of a
mathematical statement. However, deciding if use of $\mathbf{AC}$ in a
specific proof is unnecessary, or determining the exact portion of $\mathbf{%
AC}$ needed to carry out the proof, is not so obvious. In our opinion,
working in $\mathbf{ZF}$, leads to a better understanding of the various
mathematical notions involved in a proof. We find the following quotation of
Horst Herrlich, expressed in \cite{hher}, quite illuminating and
corroborative to the opinion expressed earlier.

\begin{quote}
Ordinarily topology is dealt with in the setting of $\mathbf{ZFC}$. Although 
$\mathbf{AC}$ is neither evidently true nor evidently false, this adherence
to $\mathbf{AC}$ seems to be based on a general belief that adoption of $%
\mathbf{AC}$ enables topologists to prove more and better theorems. Aside
from the trivial observation that no theorem $\mathbf{T}$ in $\mathbf{ZFC}$
is lost in $\mathbf{ZF}$, it simply turns into the implication $\mathbf{AC}%
\rightarrow \mathbf{T}$, which often enough can be even improved to an
equivalence $\mathbf{WC}\leftrightarrow \mathbf{T}$ for a suitable weak form 
$\mathbf{WC}$ of $\mathbf{AC}$.
\end{quote}

Various topological completeness properties have been invented in order to
generalize the definition of complete metric space to the context of
topologies. de Groot introduced in \cite{dg} two such properties. Namely,
subcompactness and countable subcompactness. The justification for the
introduction of the aforementioned notions, as he points out, is the
validity of the following theorem.

\begin{theorem}
\label{s3} \cite{dg} $(\mathbf{ZFC})$ Let $\mathbf{X}=(X,T)$ be a metrizable
space. The following properties are equivalent:\newline
(i) $\mathbf{X}$ is countably subcompact;\newline
(ii) $\mathbf{X}$ is subcompact;\newline
(iii) $\mathbf{X}$ is topologically complete.
\end{theorem}

Regarding Theorem \ref{s3}, it is straightforward to see that the
implication $(ii)\rightarrow (i)$ holds true in $\mathbf{ZF}$, but the
status of the remaining implications is unknown. The proofs given in \cite%
{dg} require some weak forms of the axiom of choice, such as $\mathbf{CAC}$,
in several places. The following web of implications\TEXTsymbol{\backslash}%
non-implications summarizes the established relations, in this project,
between the notions sited in Theorem \ref{s3}.
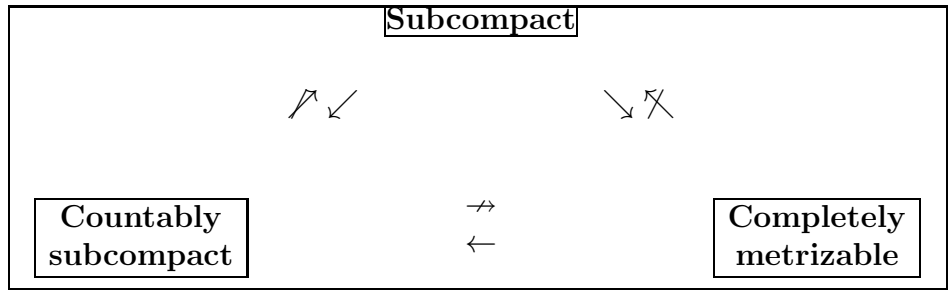
\medskip \bigskip

\begin{center}
\frame{$%
\begin{array}{ccccc}
&  & \frame{$\mathbf{Subcompact}$} &  &  \\ 
&  &  &  &  \\ 
& \text{$\not\nearrow $}\swarrow &  & \searrow \text{$\not\nwarrow $} &  \\ 
&  &  &  &  \\ 
\begin{array}{c}
\\ 
\frame{$%
\begin{array}{c}
\mathbf{Countably} \\ 
\mathbf{subcompact}%
\end{array}%
$}%
\end{array}
&  & 
\begin{array}{c}
\nrightarrow \\ 
\leftarrow%
\end{array}
&  & 
\begin{array}{c}
\\ 
\frame{$%
\begin{array}{c}
\mathbf{Completely} \\ 
\mathbf{metrizable}%
\end{array}%
$}%
\end{array}%
\end{array}%
$}\bigskip
\end{center}

In particular, in Theorem \ref{t6} we show that $(ii)\rightarrow (iii)$ and $%
(iii)\rightarrow (i)$ are valid in $\mathbf{ZF}$ and, in Theorem \ref{t3},
we show that each of the following non-implications $(i)\nrightarrow
(ii),(i)\nrightarrow (iii)$ and $(iii)\nrightarrow (ii)$\ is consistent with 
$\mathbf{ZF}$.\smallskip 

de Groot has established in \cite{dg} p. 762 the following $\mathbf{ZFC}$
characterization of compact Hausdorff spaces:

\begin{itemize}
\item $(D)$ A Hausdorff space $\mathbf{X}=(X,T)$ is compact iff it is
Tychonoff and subcompact relative to $T$.
\end{itemize}

As expected, $(D)$ is not a theorem of $\mathbf{ZF}$. It is consistent with $%
\mathbf{ZF}$ the existence of non-Tychonoff compact Hausdorff spaces, see
e.g., Example 2.4 p. 81 in \cite{gt}. As a by-product of $(D),$ if we
restrict to the class of metrizable spaces, we get the following
characterization of compactness:

\begin{itemize}
\item $(A)$ A metrizable space $\mathbf{X}=(X,T)$ is compact iff it is
subcompact relative to $T$.
\end{itemize}

\noindent Since there are complete, non-compact metric spaces, it follows
from Theorem \ref{s3} and $(A)$ that the notion \textquotedblleft subcompact
with respect to the base of all open sets\textquotedblright\ is strictly
stronger than subcompactness. The most natural question which pops up at
this point is\medskip\ 

\textbf{Question 1}. Is $(A)$ a theorem of $\mathbf{ZF}$?\medskip

Of course a compact metrizable space is subcompact with respect to any base
of open sets. So, Question 1 actually concerns the converse of $(A)$. In the
forthcoming Theorem \ref{t1} we show, in $\mathbf{ZF}$, that a metrizable
space is countably compact iff it is countably subcompact with respect to
the base of all open sets. Since a countably compact metrizable space is
compact in $\mathbf{ZF}+\mathbf{CAC}$, see e.g., \cite{kn}, it follows that $%
(A)$ is a theorem of $\mathbf{ZF}+\mathbf{CAC}$. In Theorem \ref{t9} we
answer Question 1 in the affirmative. So, subcompactness of metrizable
spaces with respect to the family of all open sets is the strongest of all
forms of compactness of metrizable spaces in $\mathbf{ZF}$, see e.g. \cite%
{kkerem}. 

The next theorem is from \cite{dg} and concerns products of subcompact
spaces. It shows, in $\mathbf{ZFC}$, that subcompactness is an invariant for
the forming of topological products.

\begin{theorem}
\label{R1} $(\mathbf{ZFC})$ \cite{dg} Let $\{\mathbf{X}_{i}=(X_{i},T_{i}):i%
\in I\}$ be a family of subcompact $T_{3}$ spaces, $\mathcal{B}=\{\mathcal{B}%
_{i}:i\in I\}$ a family of sets such that for every $i\in I$, $\mathcal{B}%
_{i}$ is a subcompact base for $\mathbf{X}_{i}$, and $\mathbf{X}%
=\dprod\limits_{i\in I}\mathbf{X}_{i}$ be their product. Then $\mathbf{X}$
is subcompact with respect to the standard base $\mathcal{C}$ generated by
the family $\mathcal{B}$.
\end{theorem}

The question which arises now is whether Theorem \ref{R1} holds in $\mathbf{%
ZF}$. We show in the forthcoming Theorem \ref{R2} that the answer, as
expected, is in the negative.\medskip

We list the following known results here for future reference.

\begin{theorem}
\label{s1} \cite{nike} $(\mathbf{ZF})$ Let $\mathbf{X}=(X,T)$ be a
topological space. The following are equivalent: \newline
(i) $\mathbf{X}$ is countably lightly compact;\newline
(ii) Every countable open filterbase $\mathcal{F}$ of $\mathbf{X}$ has a
point of adherence $(\dbigcap \{\overline{F}:F\in \mathcal{F}\}\neq
\emptyset )$;\newline
(iii) $\mathbf{X}$ has no countably infinite pairwise disjoint locally
finite family of open sets.\newline
\end{theorem}

\begin{theorem}
\label{s2} \cite{kker} $(\mathbf{ZF})$ A topological space is countably
compact iff\ is countably lightly compact.
\end{theorem}

\begin{theorem}
\label{s0} \cite{yi} $(\mathbf{ZF})$ Let $\mathbf{X}=(X,T)$ be a topological
space and $\mathcal{B}$ be an open base of $\mathbf{X}$. Then, $\mathbf{X}$
is countably $\mathcal{B}$-subcompact iff every regular $\mathcal{B}$%
-sequence has a non-empty intersection.
\end{theorem}

\begin{theorem}
\label{s0000} $(\mathbf{ZF})$ \cite{w} A $G_{\delta }$ subspace of a
completely metrizable space is completely metrizable.
\end{theorem}

\begin{theorem}
\label{s000} \cite{br} If there exists a Dedekind finite subset of $\mathbb{R%
}$ then there exists a dense one also.
\end{theorem}

The following result shows that for regular spaces their total filters
coincide with those whose intersections of the closures of their members are
non-empty.

\begin{proposition}
\label{p0} $(\mathbf{ZF})$\ An open filter $\mathcal{F}$ of a topological
space $\mathbf{X}=(X,T)$ is total iff $\dbigcap \{\overline{F}:F\in \mathcal{%
F}\}=\emptyset $.
\end{proposition}

\begin{proof}
($\rightarrow $) This is straightforward.

($\leftarrow $) Let $\mathcal{F}$ be an open filter of $\mathbf{X}$ with $%
\dbigcap \{\overline{F}:F\in \mathcal{F}\}=\emptyset $. Then, for every $%
x\in X$, there is $F\in \mathcal{F}$ with $x\in \overline{F}^{c}$. Since $%
\mathbf{X}$ is regular, there exist a neighborhood $V$ of $x$ with $%
\overline{V}\subseteq \overline{F}^{c}$. Therefore, $F\subseteq \overline{F}%
\subseteq \overline{V}^{c}$ and $\overline{V}^{c}\in \mathcal{F}$, meaning
that $\mathcal{F}$ is total.\medskip 
\end{proof}

We point out here that for every topological space $\mathbf{X}=(X,T)$ if $%
\mathcal{F}$ is a total filter of $\mathbf{X}$ then $\dbigcap \{\overline{F}%
:F\in \mathcal{F}\}=\emptyset $. The following example shows that the
converse of Proposition \ref{p0} is not true in case $\mathbf{X}$ is not
regular. 

\begin{example}
\label{e1} $(\mathbf{ZF})$ Let $T_{\mathbb{Z}}$ be the topology on the set
of all integers $\mathbb{Z}$ in which every point of $\mathbb{Z}\backslash
\{0\}$ is isolated while neighborhoods of $0$ are all cofinite subsets of $%
\mathbb{Z}$ including $0$. Let $\mathbf{Y}$ be the product of the discrete
space $\mathbb{N}\ $with $(\mathbb{Z},T_{\mathbb{Z}})$ and $X=\{\infty
\}\cup \mathbb{N}\times \mathbb{Z}$. Topologize $X$ by declaring
neighborhoods of $\mathbb{N}\times \mathbb{Z}$ to be the old ones whereas
basic neighborhoods of $\infty $ are all sets of the form $\{\infty \}\cup O$
where, $O$ is a subset of $\mathbb{N}\times \{-i:i\in \mathbb{N}\}$ such
that for all but finitely many $n\in \mathbb{N},O$ contains all but finitely
many members of the $n$-th copy of the set of negative integers, i.e.%
\[
|\{n\}\times \{-i:i\in \mathbb{N}\}\backslash O|<\aleph _{0}\text{.}
\]%
We leave it as an easy exercise for the reader to verify that $\mathbf{X}$
is a (non-compact) Hausdorff, non-regular space (the closed set $%
\{(0,n):n\in \mathbb{N}\}$ of $\mathbf{X}$ and the point $\infty \in X$
cannot be separated by open sets). Clearly, for every $k\in \mathbb{N},$%
\[
F_{k}=\dbigcup \{\{n\}\times \omega :n\geq k\}
\]%
is a clopen subset of $\mathbf{X}\ $with 
\[
\dbigcap \{\overline{F_{k}}:k\in \mathbb{N}\}=\dbigcap \{F_{k}:k\in \mathbb{N%
}\}=\emptyset \text{.}
\]%
Let $\mathcal{F}$ be the open filter of $\mathbf{X}$ generated by the open
filterbase $\{F_{k}:k\in \mathbb{N}\}$ of $\mathbf{X}$, i.e. 
\[
\mathcal{F}=\{F\subseteq X:F\text{ is open in }\mathbf{X}\text{ and for some 
}k\in \mathbb{N},F_{k}\subseteq F\}\text{.}
\]%
Clearly, $\dbigcap \{\overline{F}:F\in \mathcal{F}\}=\emptyset $. We claim
that for every neighborhood $V_{\infty }$ of $\infty $ and every $k\in 
\mathbb{N},\overline{V_{\infty }}^{c}\nsupseteq F_{k}$. To this end, fix $%
t\in \mathbb{N}$ such that for all $n\geq t,|\{n\}\times \{-i:i\in \mathbb{N}%
\}\backslash V_{\infty }|<\aleph _{0}$. Then, for all $n\geq t$ and for all $%
n\in \mathbb{N},(n,0)\in \overline{V_{\infty }}\cap F_{k}\neq \emptyset $
meaning that $F_{k}\nsubseteq \overline{V_{\infty }}^{c}$. Hence, for every $%
F\in \mathcal{F},\overline{V_{\infty }}\cap F\neq \emptyset $. So, $%
\overline{V_{\infty }}^{c}\nsupseteq F$ and $\mathcal{F}$ is not total
filter.\medskip 
\end{example}

Clearly, an infinite $T_{1}$ space has open filters with empty intersection,
e.g., the open filter generated by the family of all cofinite sets. We show
next that no compact regular space has total filters. In fact, a regular
space is compact iff it has no total filters.

\begin{theorem}
\label{t8}$\mathbf{(ZF)}$ A regular space $\mathbf{X}=(X,T)$ is compact iff
it has no total filters.
\end{theorem}

\begin{proof}
Fix a regular space $\mathbf{X}=(X,T)$.

($\rightarrow $) We show that $\mathbf{X}$ has no total filters. Assume the
contrary and let $\mathcal{F}$ be a total filter of $\mathbf{X}$. Fix, by
the compactness of $\mathbf{X},x\in \dbigcap \{\overline{F}:F\in \mathcal{F}%
\}$. By our hypothesis, there exists a neighborhood $V$ of $x$ such that $%
\overline{V}^{c}\in \mathcal{F}$. Since $\overline{\overline{V}^{c}}%
\subseteq V^{c},$ it follows that $x\notin \dbigcap \{\overline{F}:F\in 
\mathcal{F}\}$. Contradiction!

($\leftarrow $) We show that $\mathbf{X}$ is compact. To this end, we assume
the contrary and fix an open cover $\mathcal{U}$ of $\mathbf{X}$ without a
finite subcover. Let $\mathcal{V}=\{V\in T:\overline{V}\subseteq U$ for some 
$U\in \mathcal{U}\}$. By the regularity of $\mathbf{X}$, it follows easily
that $\mathcal{V}$ is an open cover of $\mathbf{X}$ such that the closed
cover $\{\overline{V}:V\in \mathcal{V}\}$ of $\mathbf{X}$ has no finite
subcover. It is easy to see that $\{\overline{V}^{c}:V\in \mathcal{V}\}$ is
a family of open sets of $\mathbf{X}$ with the fip. Let $\mathcal{F}$ be the
open filter of $\mathbf{X}$ generated by $\{\overline{V}^{c}:V\in \mathcal{V}%
\}$. Since $\mathcal{V}$ covers $X$, it follows that for every $x\in X$,
there is a $V\in \mathcal{V}$ with $x\in V$ and $\overline{V}^{c}\in 
\mathcal{F}$, meaning that $\mathcal{F}$ is total and contradicting our
hypothesis. Therefore, $\mathbf{X}$ is compact as required 
\end{proof}

\section{Main results}

Our first result in this section shows, in $\mathbf{ZF}$, that in the class
of all metrizable spaces, countable subcompactness relative to the base of
all open sets is equivalent to countable compactness, as well as to
countable light compactness.

\begin{theorem}
\label{t1}$(\mathbf{ZF})$ Let $\mathbf{X}=(X,T)$ be a metrizable space and $%
d $ be a metric on $X$ with $T=T_{d}$. The following are equivalent:\newline
(i) $\mathbf{X}$ is countably subcompact with respect to $T_{d}$;\newline
(ii) $\mathbf{X}$ is countably lightly compact;\newline
(iii) $\mathbf{X}$ is countably compact.
\end{theorem}

\begin{proof}
Fix a metric $d$ on $X$ with $T=T_{d}$. It suffices, in view of Theorem \ref%
{s2}, to show (i) $\leftrightarrow $ (ii).\smallskip\ 

(i) $\rightarrow $ (ii) Assume the contrary and fix a countable, pairwise
disjoint, locally finite family $\mathcal{U}_{0}=\{U_{0i}:i\in \mathbb{N}\}$
of open subsets of $\mathbf{X}$. We are going to construct inductively
families $\mathcal{U}_{n}=\{U_{ni}:i\geq n\},n\in \mathbb{N}$ of open sets
of $\mathbf{X}$ such that for all $i\geq n,\overline{U_{ni}}\subseteq
U_{(n-1)i}$. For every non-empty open set $U$ of $\mathbf{X}$ let 
\begin{equation}
t_{U}=\min \{n\in \mathbb{N}:\{x\in U:d(x,U^{c})>1/n\}\neq \emptyset \}\text{%
.}  \label{1}
\end{equation}%
We begin the induction by letting for $n=1$ and every $i\in \mathbb{N}$, 
\[
U_{1i}=\{x\in U_{0i}:d(x,U_{0i}^{c})>1/t_{U_{0i}}\}\text{.} 
\]%
Since $\mathcal{U}_{0}$ is locally finite, it follows that $\mathcal{U}%
_{1}=\{U_{1i}:i\geq 1\}$ is a locally finite family of open sets of $\mathbf{%
X}$. Furthermore, for all $i\geq 1,\overline{U_{1i}}\subseteq U_{0i}$.

For $n=k+1$, use the induction hypothesis on $\mathcal{U}_{k}=\{U_{ki}:i\geq
k\}$ and define for every $i\geq n$, 
\[
U_{ni}=\{x\in U_{ki}:d(x,U_{ki}^{c})>1/t_{U_{ki}}\}\text{.} 
\]%
Clearly, $\mathcal{U}_{n}=\{U_{ni}:i\geq n\}$ is a locally finite family of
open sets of $\mathbf{X}$, and for all $i\geq n,\overline{U_{ni}}\subseteq
U_{ki}$ terminating the induction.\smallskip

Let $\dbigcup \mathcal{U}_{0}=F_{0}$ and for every $n\in \mathbb{N}$ put $%
F_{n}=\dbigcup \{U_{ni}:i\geq n\}$. Clearly, 
\[
\overline{F_{1}}=\overline{\dbigcup \{U_{1i}:i\geq 1\}}=\dbigcup \{\overline{%
U_{1i}}:i\geq 1\}\subseteq F_{0},
\]%
and for all $n>1$, 
\[
\overline{F_{n}}=\overline{\dbigcup \{U_{ni}:i\geq n\}}=\dbigcup \{\overline{%
U_{ni}}:i\geq n\}\subseteq \dbigcup \{U_{(n-1)i}:i\geq n-1\}=F_{n-1}\text{.}
\]%
Hence, $\mathcal{F}=\{F_{n}:n\in \omega \}$ is a regular $T_{d}$-sequence.
Thus, by the subcompactness of $\mathbf{X}$ with respect to $T_{d}$, and
Theorem \ref{s0}, it follows that $\dbigcap \mathcal{F}\neq \emptyset $.
Since $\mathcal{U}_{0}$ is pairwise disjoint and $F_{0}\in \mathcal{F}$, it
follows that for every $x\in \dbigcap \mathcal{F},x\in U_{0i}$ for some $%
i\in \mathbb{N}$. Since $x\notin F_{i+1}$, we conclude that $x\notin
\dbigcap \mathcal{F}$. Contradiction! Thus, $\mathcal{U}_{0}$ is finite and $%
\mathbf{X}$ is countably lightly compact as required.\smallskip 

(ii) $\rightarrow $ (i) Fix a countable regular filterbase $\mathcal{F}$ of
open subsets of $\mathbf{X}$. By our hypothesis and Theorem \ref{s1}, it
follows that $\dbigcap \{\overline{F}:F\in \mathcal{F}\}\neq \emptyset $.
Since $\mathcal{F}$ is regular it follows that $\dbigcap \mathcal{F}\neq
\emptyset $, and $\mathbf{X}$ is subcompact as required.\smallskip
\end{proof}

As a corollary to Theorem \ref{t1} we get:

\begin{corollary}
\label{c2} $(\mathbf{ZF})$ A second countable metrizable space $\mathbf{X}%
=(X,T)$ is compact iff it is countably subcompact with respect to $T$.
\end{corollary}

\begin{proof}
It suffices, in view of Theorem \ref{t1}, to show that a countably compact
metric space is compact. This follows at once from Corollary 20 in \cite%
{kkerem}.\medskip
\end{proof}

Taking into consideration Theorem \ref{t1}, one may ask whether the
statement \textquotedblleft every countably compact metrizable space is
subcompact with respect to the base of all open sets\textquotedblright\ is a
theorem of $\mathbf{ZF}$. We observe next that this is not the case by
establishing that it implies $\mathbf{IWDI}$, a statement whose negation is
known to be consistent with $\mathbf{ZF}$, see e.g., Form 82 in \cite{hr}.

\begin{theorem}
\label{t4}The proposition \textquotedblleft every countably compact
metrizable space is subcompact with respect to the base of all open
sets\textquotedblright\ implies $\mathbf{IWDI}$.
\end{theorem}

\begin{proof}
Assume the contrary and let $X$ be an infinite weakly Dedekind-finite set.
Let $d$ be the discrete metric on $X$. By our hypothesis, $\mathbf{X}$ has
no denumerable open covers. So, $\mathbf{X}$ is trivially countably compact.
Therefore, by our hypothesis, $\mathbf{X}$ is subcompact with respect to $%
T_{d}$. Since the family $\mathcal{F}$ of all cofinite subsets of $\mathbf{X}
$ is trivially a regular filter with $\dbigcap \mathcal{F}=\emptyset $, we
arrive at a contradiction. Hence, $\mathbf{IWDI}$ holds true as
required.\medskip\ 
\end{proof}

Next we answer Question 1 in the affirmative.

\begin{theorem}
\label{t9}$\mathbf{(ZF)}$ Let $\mathbf{X}=(X,T)$ be a metrizable space. The
following are equivalent:\newline
(i) $\mathbf{X}$ is compact;\newline
(ii) $\mathbf{X}$ has no total filters; \newline
(iii) $\mathbf{X}$ is subcompact relative to $T$.
\end{theorem}

\begin{proof}
Let $\mathbf{X}=(X,T)$ be a metrizable space and fix a metric $d$ on $X$
with $T=T_{d}$.

(i) $\rightarrow $ (ii) This follows from Theorem \ref{t8} and the fact that
metrizable spaces are regular.\smallskip

(ii) $\rightarrow $ (iii) Fix a regular filterbase $\mathcal{F}$ of $\mathbf{%
X}$. For our convenience we may assume that $\mathcal{F}$ is also an open
filter of $\mathbf{X}$. We show that $\dbigcap \mathcal{F\neq \emptyset }$.
Assume the contrary and let $\dbigcap \mathcal{F=\emptyset }$. We claim that 
$\mathcal{F}$ is a total filter. Since $\mathcal{F}$ is regular, 
\begin{equation}
\dbigcap \mathcal{F}=\dbigcap \{\overline{F}:F\in \mathcal{F}\}\text{.}
\label{10}
\end{equation}%
Since $\mathbf{X}$ is regular, by Proposition \ref{p0} and (\ref{10}), it
follows that $\mathcal{F}$ is a total filter. Contradiction!\smallskip 

(iii) $\rightarrow $ (i) Assume the contrary and fix a family $\mathcal{G}$
of closed subsets of $\mathbf{X}$ with the fip such that $\dbigcap \mathcal{G%
}=\emptyset $. Without loss of generality we may assume that $\mathcal{G}$
is closed under finite intersections. Let 
\[
\mathcal{F}=\{U\in T:\text{for some }G\in \mathcal{G},G\subseteq U\}\text{.} 
\]%
We claim that $\mathcal{F}$ is a regular filterbase of $\mathbf{X}$. To see
that $\mathcal{F}$ is a filterbase of $\mathbf{X}$, fix $F_{1},F_{2}\in 
\mathcal{F}$ and let $G_{1},G_{2}\in \mathcal{G}$ satisfy $G_{1}\subseteq
F_{1}$ and $G_{2}\subseteq F_{2}$. Since $G_{1}\cap G_{2}\in \mathcal{G}$
and $G_{1}\cap G_{2}\subseteq F_{1}\cap F_{2}$, it follows that $F_{1}\cap
F_{2}\in \mathcal{F}$. To see that $\mathcal{F}$ is a regular filterbase,
fix $F\in \mathcal{F}$ and let $G\in \mathcal{G}$ satisfy $G\subseteq F$.
Since $\mathbf{X}$ is $T_{4}$ in $\mathbf{ZF}$, it follows that there exist
disjoint open sets $U,V$ of $\mathbf{X}$ with $G\subseteq U$ and $%
F^{c}\subseteq V$. It follows that $U\in \mathcal{F}$ and $V^{c}$ is closed.
Clearly, $G\subseteq U\subseteq \overline{U}\subseteq V^{c}\subseteq F$.
Hence, $\mathcal{F}$ is regular filterbase.

Fix, by our hypothesis, $x\in \dbigcap \mathcal{F}$. If $x\notin \dbigcap 
\mathcal{G}$ then for some $G\in \mathcal{G}$, $x\notin G$. Hence, $%
G\subseteq \{x\}^{c}$ and $\{x\}^{c}\in \mathcal{F}$, meaning that $x\notin
\dbigcap \mathcal{F}$. Contradiction! Thus, $x\in \dbigcap \mathcal{G}$ and $%
\mathbf{X}$ is compact as required.\medskip 
\end{proof}

In $\mathbf{ZFC}$, a subcompact metrizable space need not be countably
compact, hence by Theorem \ref{t1}, not countably subcompact with respect to
the base of all open sets. Indeed, $\mathbb{N}$ with the discrete metric is
subcompact ($\mathcal{B}=\{\{n\}:n\in \mathbb{N}\}$ is a subcompact base)
but $\mathbb{N}$ is not countably compact. Next we show, in $\mathbf{ZF}$,
that every subcompact metrizable space is completely metrizable and, every
completely metrizable space is countably subcompact.

\begin{theorem}
\label{t6} $(\mathbf{ZF})$ (i) Every subcompact metric space is a $G_{\delta
}$ set in any of its completions.\newline
In particular, every subcompact metric space is completely metrizable.%
\newline
(ii) Every completely metrizable space is countably subcompact.
\end{theorem}

\begin{proof}
(i)\ A straightforward modification of the proof given in \cite{dg} p. 763
shows that we can dispense with any use of $\mathbf{AC}$. For the reader's
convenience we supply all the details below. Fix a metric space $\mathbf{X}%
=(X,d)$ having a subcompact base $\mathcal{B}$ and let $\mathbf{Y}=(Y,\sigma
)$ be a completion of $\mathbf{X}$. Clearly, every $B\in \mathcal{B}$ can be
expressed as $X\cap O$ for some open set $O$ of $\mathbf{Y}$ such that $%
\delta (B)=\delta (O)$ ($B$ is dense in $O$). Let 
\[
\mathcal{U}=\{U\in T_{\sigma }:U\cap X=B\mathcal{\ }\text{\ for some }B\in 
\mathcal{B}\}\text{.}
\]%
For every $n\in \mathbb{N}$ let 
\[
\mathcal{U}_{n}=\{U\in \mathcal{U}:\delta (U)<1/n\}\text{,}
\]%
and $U_{n}=\dbigcup \mathcal{U}_{n}$. Clearly, for every $n\in \mathbb{N}%
,X\subseteq U_{n}$. We show that 
\[
X=\dbigcap \{U_{n}:n\in \mathbb{N}\}.
\]%
Since $\mathcal{B}$ is a base for $\mathbf{X}$, it follows that for every $%
x\in X$ and every $n\in \mathbb{N}$, there exists a $B\in \mathcal{B}$ such
that $x\in B$ and $\delta (B)<1/n$. Hence, there exists $U\in \mathcal{U}$
such that $x\in U$ and $\delta (U)<1/n$. Therefore, for every $n\in \mathbb{N%
},x\in U_{n}$. Hence, 
\begin{equation}
X\subseteq \dbigcap \{U_{n}:n\in \mathbb{N}\}\text{.}  \label{3}
\end{equation}%
To see the other direction of the inclusion, fix $y\in \dbigcap \{U_{n}:n\in 
\mathbb{N}\}$. Since for every $n\in \mathbb{N}$, $y\in U_{n}$ it follows
that there exists $U\in \mathcal{U}_{n}$ with $y\in U$. Hence, 
\begin{equation}
\dbigcap \{U\in \mathcal{U}:y\in U\}=\{y\}\text{.}  \label{5}
\end{equation}%
It is easy to see that 
\begin{equation}
\text{for every }U\in \mathcal{U}\text{ with }y\in U\text{ there exists }%
V\in \mathcal{U}\text{ with }y\in V,\overline{V}\subseteq U\text{.}
\label{6}
\end{equation}%
\ Let $\mathcal{F}=\{B\in \mathcal{B}:B=U\cap X$ for some $U\in \mathcal{U}$
with $y\in U\}$. We claim that $\mathcal{F}$ is a regular filterbase of $%
\mathcal{B}$. To see this, fix $B\in \mathcal{F}$ and let $U\in \mathcal{U}$
satisfy $y\in U$ and $B=U\cap X$. By (\ref{6}), there exists $V\in \mathcal{U%
}$ with $y\in V$ and $\overline{V}\subseteq U$. Let $B_{V}=V\cap X\in 
\mathcal{B}$. We have: 
\[
\overline{B_{V}}^{X}\subseteq \overline{B_{V}}=\overline{V}\subseteq U\text{.%
}
\]%
Therefore, for every $x\in \overline{B_{V}}^{X},x\in U\cap X=B$, meaning
that $\overline{B_{V}}^{X}\subseteq B$. Hence, $\mathcal{F}$ is a regular
filterbase of $\mathcal{B}$ as claimed. By the subcompactness of $\mathcal{B}
$ and (\ref{5}) we get 
\[
\emptyset \neq \dbigcap \mathcal{F}\subseteq \dbigcap \{U\in \mathcal{U}%
:y\in U\}=\{y\}\text{.}
\]%
Therefore, $\dbigcap \mathcal{F}=\{y\}$ and consequently $\dbigcap
\{U_{n}:n\in \mathbb{N}\}\subseteq X$. Hence, $X$ is a $G_{\delta }$ set in
the complete metric space $\mathbf{Y}$.\smallskip 

The second assertion follows at once from the first part and Theorem \ref%
{s0000}.\smallskip

(ii) Fix a complete metric space $\mathbf{X}=(X,d)$ and let $X_{1},X_{2}$ be
the sets of all limit and isolated points of $\mathbf{X}$ respectively. Let $%
\mathcal{B}=\mathcal{B}_{1}\cup \mathcal{B}_{2}\cup \mathcal{B}_{3}$, where 
\[
\mathcal{B}_{1}=\{B(x,1/n):n\in \mathbb{N},x\in X_{1}\text{\ and }B(x,1/n)%
\text{ has no other center }y\neq x\}, 
\]%
\[
\mathcal{B}_{2}=\{B(x,1/n)\backslash \{y\}:n\in \mathbb{N},x,y\in
X_{1},x\neq y,B(x,1/n)=B(y,1/n)\} 
\]%
and%
\[
\mathcal{B}_{3}=\{\{x\}:x\in X_{2}\}. 
\]%
It is straightforward to see that $\mathcal{B}$ is a base of $\mathbf{X}$.

We claim that $\mathcal{B}$ is countably subcompact. To this end fix, in
view of Theorem \ref{s0}, a regular $\mathcal{B}$\textit{-}sequence $%
\mathcal{F}=\{F_{n}:n\in \mathbb{N}\}$. If $F=\{x\}$ for some $F\in \mathcal{%
F}$, then $\dbigcap \mathcal{F}=F$. Assume that $\mathcal{F}\subseteq 
\mathcal{B}_{1}\cup \mathcal{B}_{2}$. For every $n\in \mathbb{N}$ define 
\[
x_{n}=\left\{ 
\begin{array}{c}
x,\text{ if }F_{n}=B(x,1/k_{n})\in \mathcal{B}_{1} \\ 
y\text{, if }F_{n}=B(t,1/k_{n})\backslash \{y\}\in \mathcal{B}_{2}%
\end{array}%
\right. . 
\]%
We consider the following two cases:

(a) $\lim_{n\rightarrow \infty }1/k_{n}=0$. In this case, it follows that $%
(x_{n})_{n\in \mathbb{N}}\ $is a Cauchy sequence of $\mathbf{X}$. Hence, by
the completeness of $\mathbf{X}$, $(x_{n})_{n\in \mathbb{N}}$ converges to
some $x\in X$. It is straightforward to see that $x\in \dbigcap \mathcal{F}$.

(b) $\lim_{n\rightarrow \infty }1/k_{n}\neq 0$. Since $F_{1}\supseteq
F_{2}\supseteq ,...$, it follows that $(1/k_{n})_{n\in \mathbb{N}}$ is a
decreasing sequence. Hence, for some $n_{0}\in \mathbb{N}$, $k_{n}=k_{n_{0}}$
for all $n\geq n_{0}$. We claim that for all $n\geq n_{0}$, $x_{n_{0}}\in 
\overline{F_{n}}$. To this end, fix $n>n_{0}$. Since, $F_{n}\subseteq
F_{n_{0}}$ it follows that $d(x_{n},x_{n_{0}})<1/k_{n}$. Since $x_{n_{0}}\in
X_{1},$ it follows that $x_{n_{0}}\in \overline{F_{n}}$. Therefore, $%
x_{n_{0}}\in \dbigcap \mathcal{F}$.

From cases (a)\ and (b)\ it follows that $\dbigcap \mathcal{F}\neq \emptyset 
$, and $\mathbf{X}$ is countably subcompact as required.\medskip
\end{proof}

In view of Theorem \ref{t1}, a metrizable countably subcompact with respect
to the base of all open sets is completely metrizable in $\mathbf{ZF}$. We
show next that this is not the case if we only assume the space to be
subcompact.

\begin{theorem}
\label{t3} Each of the following statements:\newline
(i) Every metrizable, countably subcompact topological space is completely
metrizable,\newline
(ii) every countably subcompact metrizable space is subcompact, \newline
(iii) every completely metrizable space is subcompact \newline
implies $\mathbf{IDI}(\mathbb{R})$.\newline
In particular, non of (i)-(iii) is a theorem of $\mathbf{ZF}$.
\end{theorem}

\begin{proof}
Assume the contrary and fix an infinite Dedekind finite subset $D$ of $%
\mathbb{R}$. Without loss of generality we may assume that $D\cap \mathbb{Q}%
=\emptyset $. By Theorem \ref{s000} we may assume that $D$ is dense in $%
\mathbb{R}$.\smallskip\ 

(i) Clearly, 
\[
\mathcal{B}=\{(x,y)\cap \mathbb{Q}:x,y\in D,x<y\} 
\]%
is a base for $\mathbb{Q}$ endowed with the usual (Euclidean) metric $|.|$.
Since $D\times D$ is Dedekind finite, it follows that $\mathcal{B}$ has no
denumerable subsets. Thus, $\mathbb{Q}$ is trivially countably subcompact
with respect to $\mathcal{B}$. Hence, by our hypothesis $\mathbb{Q}$ is
completely metrizable. However, $\mathbb{Q}$ is not completely metrizable
(for every metric $d$ on $\mathbb{Q}$ with $T_{d}=T_{|.|}$, $\mathbb{Q}\ $%
can be expressed as a union of countably many closed nowhere dense sets.
Therefore, $\mathbb{Q}$ is not Baire. Since separable completely metrizable
spaces are Baire in $\mathbf{ZF}$, see e.g. \cite{bru}, it follows that $%
\mathbb{Q}$ is not completely metrizable). Contradiction!\smallskip

(ii), (iii) It is easy to see that 
\[
\mathcal{B}=\{(x,y)\cap D:x,y\in D,x<y\} 
\]%
is a base for the subspace $\mathbf{D}$ of $\mathbb{R}$ endowed with the
usual metric. Since $D$ is Dedekind finite, $\mathbf{D}$ is countably
subcompact (with respect to $\mathcal{B}$) and complete. Hence, by our
hypotheses in (ii)\ and (iii), $\mathbf{D}$ is subcompact. By Theorem \ref%
{t6} (i), $D$ is $G_{\delta }$ in its completion $\mathbb{R}$. Fix a family $%
\{U_{n}:n\in \mathbb{N}\}$ of open sets of $\mathbb{R}$ with 
\[
D=\dbigcap \{U_{n}:n\in \mathbb{N}\} 
\]%
and define a map $f:D\rightarrow \mathbb{R}^{\omega }$ by requiring: 
\[
f(x)=(x,1/d(x,U_{1}^{c}),1/d(x,U_{2}^{c}),...)\text{.} 
\]%
Clearly, for every $n\in \mathbb{N}$, the function $h_{n}:\mathbb{R}%
\rightarrow \mathbb{R}$, $h_{n}(x)=d(x,U_{n}^{c})$ is strictly positive on $%
U_{n}$ and continuous. Hence, $f$ is well defined, continuous and $1:1$.
Since, $f^{-1}:f(D)\rightarrow D$ coincides with the restriction of the
projection $\pi _{0}$ to $f(D),$ it follows that $f^{-1}$ is continuous.
Therefore, $f:D\rightarrow f(D)$ is a homeomorphism.

We claim that $f(D)$ is closed in $\mathbb{R}^{\omega }$. To this end, fix $%
y=(a_{0},a_{1},...,a_{n},...)\in \mathbb{R}^{\omega }\backslash f(D).$ We
consider the following cases:\smallskip

(a) $a_{0}\in D$. In this case there exists $k\in \mathbb{N}$, such that $%
1/d(a_{0},U_{k}^{c})\neq a_{k}$. Since $1/h_{k}$ is continuous, it follows
that for $\varepsilon =|1/d(a_{0},U_{k}^{c})-a_{k}|$ there is a $\delta >0$
such that for all $x\in (a_{0}-\delta ,a_{0}+\delta ),$%
\[
|1/d(a_{0},U_{k}^{c})-1/d(x,U_{k}^{c})|<\varepsilon /3. 
\]%
Therefore, 
\[
V=(a_{0}-\delta ,a_{0}+\delta )\times \mathbb{R}\times ,...,\times \mathbb{R}%
\times (a_{k}-\varepsilon /3,a_{k}+\varepsilon /3)\times \mathbb{R}\times
,... 
\]%
is a neighborhood of $y$ avoiding $f(D)$.\smallskip

(b) $a_{0}\notin D$. In this case there exists $k\in \mathbb{N}$, such that $%
a_{0}\in U_{k}^{c}$. Since, $\lim_{x\rightarrow
a_{0}}1/d(x,U_{k}^{c})=\infty $, it follows that for $M=|a_{k}|+1$, there
exists a $\delta >0$ such that for every $x\in (a_{0}-\delta ,a_{0}+\delta
),1/d(x,U_{k}^{c})>M$. Then, 
\[
U=(a_{0}-\delta ,a_{0}+\delta )\times \mathbb{R}\times ,...,\times \mathbb{R}%
\times (a_{k}-1/2,a_{k}+1/2)\times \mathbb{R}\times ,... 
\]%
is a neighborhood of $y$ included in $f(D)^{c}$.\smallskip

From (a) and (b) it follows that $f(D)$ is a closed subset of $\mathbb{R}%
^{\omega }$ as claimed.

In \cite{kw} it has been shown, in $\mathbf{ZF}$, that the family of all
non-emtpy closed subsets of $\mathbb{R}^{\omega }$ has a choice set. Since $%
\mathbb{R}^{\omega }$ is separable, hence second countable, it follows that
its subspace $f(D)$ is also second countable. Therefore $f(D),$ and
consequently $\mathbf{D}$, is separable, contradicting the fact that $D$ is
Dedekind finite.\smallskip

The second assertion follows from the fact that in Cohen's basic model $%
\mathcal{M}$1 in \cite{hr}, $\mathbf{IDI(}\mathbb{R}\mathbf{)}$
fails.\medskip
\end{proof}

\noindent \textbf{Remark 1}. One can easily adopt the proof of part (ii) of
Theorem \ref{t3} to get a $\mathbf{ZF}$ proof of Theorem \ref{s0000} ($%
\mathbb{R}^{\omega }$ is completely metrizable in $\mathbf{ZF}$).\medskip 

Next, we show that $\mathbf{AC}$ is equivalent to the assertion that
subcompactness is an invariant for the forming of topological products.

\begin{theorem}
\label{R2}The following are equivalent: (i) $\mathbf{AC}$;\newline
(ii) for every family $\{\mathbf{X}_{i}:i\in I\}$ of subcompact $T_{3}$
spaces, for every family $\{\mathcal{B}_{i}:i\in I\}$ such that for every $%
i\in I$, $\mathcal{B}_{i}$ is a subcompact base for $\mathbf{X}_{i}$, the
Tychonoff product $\mathbf{X}=\dprod\limits_{i\in I}\mathbf{X}_{i}$ is
subcompact with respect to the standard base $\mathcal{B}$ of $\mathbf{X}$
generated by the family $\{\mathcal{B}_{i}:i\in I\}$;\newline
(iii) for every family $\{\mathbf{X}_{i}:i\in I\}$ of metrizable subcompact
spaces, for every family $\{\mathcal{B}_{i}:i\in I\}$ such that for every $%
i\in I$, $\mathcal{B}_{i}$ is a subcompact base for $\mathbf{X}_{i}$, the
Tychonoff product $\mathbf{X}=\dprod\limits_{i\in I}\mathbf{X}_{i}$ is
subcompact with respect to the standard base $\mathcal{B}$ of $\mathbf{X}$
generated by the family $\{\mathcal{B}_{i}:i\in I\}$.
\end{theorem}

\begin{proof}
(i) $\rightarrow $ (ii) follows from Theorem \ref{R1} and (ii) $\rightarrow $
(iii) is straightforward.\smallskip

(iii) $\rightarrow $ (i) Fix $\mathcal{A}=\{A_{i}:i\in I\}$ a pairwise
disjoint family of non-empty sets and let $\infty $ be a set not in $%
\dbigcup \mathcal{A}$. For every $i\in I$, let $X_{i}=A_{i}\cup \{\infty \}$
carry the discrete metric. Clearly, for every $i\in I$, $\mathcal{B}%
_{i}=\{\{x\}:x\in X\}\cup \{A_{i}\}$ is a subcompact base for $\mathbf{X}%
_{i} $. By our hypothesis, the product $\mathbf{X}=\dprod\limits_{i\in I}%
\mathbf{X}_{i}$ is subcompact with respect to the standard base $\mathcal{B}$
of $\mathbf{X}$ generated by the family $\{\mathcal{B}_{i}:i\in I\}$. It is
straightforward to verify that 
\begin{equation}
\mathcal{F}=\{\dbigcap Q:Q\text{\ is a finite non-empty subset of }\{\pi
_{i}^{-1}(A_{i}):i\in I\}\}\text{.}  \label{4}
\end{equation}%
is a regular $\mathcal{B}$-filterbase of $\mathbf{X}$. Hence, by our
hypothesis, $\dbigcap \mathcal{F}\neq \emptyset $. Clearly, any element $%
f\in \dbigcap \mathcal{F}$ is a choice function of $\mathcal{A}$.\bigskip
\end{proof}

\noindent \textsc{Kyriakos Keremedis}\newline
Department of Mathematics\newline
University of the Aegean\newline
Karlovassi, Samos 83200, Greece \newline
\emph{E-mail}: kker@aegean.gr

\end{document}